\documentclass[11pt,reqno]{amsart}
\usepackage[dvips]{graphicx,color,epsfig}
\usepackage{amssymb,amscd,latexsym,xypic}
\usepackage[mathscr]{euscript}
\DeclareFontFamily{OT1}{rsfs}{}
\DeclareFontShape{OT1}{rsfs}{n}{it}{<-> rsfs10}{}
\DeclareMathAlphabet{\curly}{OT1}{rsfs}{n}{it}

\renewcommand\O{\mathcal O}

\newcommand\PP{\mathbb P}

\newcommand\C{\mathbb C}
\newcommand\cC{\mathcal C}

\newcommand\into{\hookrightarrow}

\newcommand\ito{\ar@{^{ (}->}[r]}

\renewcommand\_{^{}_}
\newcommand\take{\backslash}

\newfont{\bigtimesfont}{cmsy10 scaled \magstep5}
\newcommand{\bigtimes}{\mathop{\lower0.9ex\hbox{\bigtimesfont\symbol2}}}
\newcommand\bull{{\scriptscriptstyle\bullet}}
\renewcommand\bar{\overline}

\newcommand\id{\operatorname{id}}

\renewcommand\hom{\curly H\!om}

\newcommand\Hilb{\operatorname{Hilb}}
\newcommand\Sym{\operatorname{Sym}}

\newcommand\beq[1]{\begin{equation}\label{#1}}
\newcommand\eeq{\end{equation}}
\newcommand\beqa{\begin{eqnarray}}
\newcommand\eeqa{\end{eqnarray}}
\newcommand\beqas{\begin{eqnarray*}}
\newcommand\eeqas{\end{eqnarray*}}

\makeatletter \@addtoreset{equation}{section} \makeatother

\newtheorem{thm}[equation]{Theorem}

\newtheorem{prop}[equation]{Proposition}

\title{A short proof of the G\"ottsche conjecture}
\author[M. Kool, V. Shende and R. P. Thomas]{Martijn Kool, Vivek Shende and Richard Thomas}

\begin{document}
\begin{abstract}
We prove that for a sufficiently ample line bundle $L$ on a surface $S$, the number
of $\delta$-nodal curves in a general $\delta$-dimensional linear system is given by
a universal
polynomial of degree $\delta$ in the four numbers $L^2,\,L.K_S,\,K_S^2$ and $c_2(S)$.

The technique is a study of Hilbert schemes of points on curves on a surface, using
the BPS calculus of \cite{PT3} and the computation of tautological integrals on Hilbert
schemes by Ellingsrud, G\"ottsche and Lehn.

We are also able to weaken the ampleness required, from G\"ottsche's $(5\delta-1)$-very
ample to $\delta$-very ample.
\end{abstract}
\maketitle


\section{Introduction}
Throughout this paper we fix a compact complex surface $S$ with very ample line bundle $L$ with no higher
cohomology. Curves $C$ in the linear system $\PP:=\PP(H^0(L))$ have arithmetic genus
$g=g(C)$, where $2g-2=L.(L+K_S)$.

Call a (possibly reducible) curve $\delta$-nodal if it is smooth away from $\delta$ points at which it looks analytically like $\{xy=0\}\subset\C^2$. For sufficiently ample $L$, $\delta$-nodal curves occur in codimension $\delta$, and it is a classical question how many such curves appear in a general $\delta$-dimensional
linear subsystem $\PP^\delta\subset\PP$: see \cite{KP2} for a history going back to 1848 and Steiner, Cayley and Salmon. For $S=\PP^2$, Ran \cite{R}, and later Caporaso-Harris \cite{CH},
found recursions that determine these Severi degrees. For general $S$ it was conjectured \cite{DFI, G, KP1, V} that they should be topological -- in fact a universal
polynomial of degree $\delta$ in the numbers $L^2,\,L.K_S,\,K_S^2$ and $c_2(S)$. This is often referred to as the
G\"ottsche conjecture; in fact G\"ottsche \cite{G} gave a partial conjectural form for the generating series (motivated by the Yau-Zaslow formula \cite{YZ} and involving quasimodular
forms) about which we have nothing to say in this paper. It is plausible that our recursive formulae could be used to identify these generating series, but it seems unlikely: the
same integrals are evaluated for $S=K3$ in \cite{KY, MPT} but by an indirect method.

There are now proofs for $\PP^2$ \cite{Ch, FM} and other surfaces \cite{Be, BL1, BL2,
Ka, Liu1, Liu2}
using wildly different techniques. A completely general algebro-geometric proof using degeneration has recently been found by Tzeng \cite{Tz}.

Our method of counting is rather simpler; it is a generalisation of the following elementary standard technique
for $\delta=1$. Suppose we have a pencil of curves, a finite number of which have one node (and the rest are smooth). The smooth curves have Euler characteristic $2-2g$, while the nodal curves have Euler characteristic $(2-2g)+1$. Thus we can determine the number of nodal curves from the topological Euler characteristic of the universal curve $\cC\to\PP^1$:
\beq{easy}
e(\cC)=(2-2g)e(\PP^1)+\#(1\text{-nodal curves}).
\eeq
Since $\cC$ is the blow up of $S$ at $L^2$ points this gives the classical formula
$$
\#(1\text{-nodal curves})=c_2(S)+L^2+2L.(L+K_S).
$$

For a curve $C$, define $e_i:=e\big(C^{[i]}\big)$, where $C^{[i]}$ is the Hilbert scheme of $i$ points on $C$. Then the point is that
\beq{bps1}
e_1-(2-2g)e_0
\eeq
is 1 when $C$ is 1-nodal, and 0 if $C$ is smooth. Summing over all the curves in the pencil gives \eqref{easy}.

Similarly for arbitrary $\delta$ there is a linear formula in the $e_i,\ 0\le i\le\delta$, which gives 1 for $\delta$-nodal curves and 0 for
all curves of geometric genus $\bar g>g-\delta$ (those which are less singular in some sense). The result -- Theorem \ref{BPS} below --
is taken from \cite{PT3}, where it was proved in the context of stable pairs (these are in some sense dual to points of the Hilbert scheme). 

Summing over $\PP^\delta$ gives a
formula for the number of $\delta$-nodal curves in terms of the Euler characteristics of the \emph{relative Hilbert schemes} of the
universal curve. In turn we compute these
Euler characteristics of relative Hilbert schemes in terms of certain tautological integrals
over the Hilbert schemes of points on the surface. These can be handled by a recursion
due to Ellingsrud, G\"ottsche and Lehn \cite{EGL}. (G\"ottsche also expressed
the counts of nodal curves in terms of integrals on the Hilbert scheme, but to which
\cite{EGL} does not apply. They are evaluated by degeneration techniques by Tzeng in
\cite{Tz}.)

\subsection*{Motivation} This paper arose from our project \cite{KT} to define the
invariants counting nodal curves in terms of virtual classes, thus extending them from
the very
ample case to more general curve classes. There are two obvious ways to do this, using
Gromov-Witten theory or stable pairs (using equivariant \emph{reduced} 3-fold virtual
classes as in \cite{MPT}). These should be related by the stable pair version \cite{PT1}
of the famous MNOP conjecture \cite{MNOP1}.
To see the $\delta$-nodal curves in GW theory we change the genus and look at stable maps from
curves of genus $g-\delta$. Via the the MNOP conjecture, the analogue in stable pairs
theory is to
allow up to $\delta$ free points to roam around the curves and use the BPS calculus
of \cite{PT3} to identify the lower genus curves.
This is essentially what we do in this paper. But we forget about the motivation, working
from the start with Hilbert schemes of points on curves. We also confine ourselves to
the classical case of very ample curves, so that no virtual cycles enter.

\subsection*{Notation}
For a reduced curve $C$ of arithmetic genus $g=g(C)$, we let $\bar g:=g(\bar C)$ be its geometric genus: the genus of its normalisation $\bar C\to C$. If $\bar C$ is disconnected
with connected components $C_i$ then its genus is $1+\sum_i(g(C_i)-1)$, which can be
negative. The total Chern class of a bundle $E$ is denoted by $c_\bull(E)$.
A line bundle $L$ on $S$ is said to be $n$-very ample if $H^0(L)\to H^0(L|_Z)$
is onto for all length-$(n+1)$ subschemes $Z\subset S$.
We use $X^{[k]}$ for the Hilbert scheme of $k$ points on any variety $X$, but $\Hilb^k(X/B)$ for the \emph{relative} Hilbert scheme of points on the fibres of a family $X\to B$.
Using the obvious projections and the universal subscheme $\mathcal Z_k\subset X\times
X^{[k]}$ we get the rank $k$ tautological bundle
$$
L^{[k]}:=\pi_{2*}\big((\pi_1^*L)|_{\mathcal Z_k}\big) \quad\mathrm{on}\ X^{[k]}.
$$

\subsection*{Acknowledgements} We would like to thank Daniel Huybrechts, Rahul Pandharipande, Dmitri Panov for useful conversations, and Samuel Boissi\`ere and Julien Grivaux for pointing out the reference \cite{EGL}.

\section{Sufficiently ample linear systems}
We begin with a description of the curves in a general sufficiently ample linear system,
strengthening a result of G\"ottsche \cite[Proposition 5.2]{G}. The bound on the geometric genus is key to our results.

\begin{prop} \label{Got}
If $L$ is $\delta$-very ample then the general $\delta$-dimensional linear system $\PP^\delta\subset\PP(H^0(L))$ contains a finite number of $\delta$-nodal curves appearing with multiplicity 1. All other curves are reduced with geometric genus $\bar g>g-\delta$.
\end{prop}

\begin{proof}
We recall some deformation theory of singularities of curves in surfaces; an excellent reference is \cite{DH}. Since everything is local we consider the germ of a reduced plane curve $C=\{f=0\}\subset\C^2$  about its singular points. This has a miniversal deformation space $\mathsf{Def}= H^0(\O_C/J)$,
where $J=(\partial f/\partial x,\partial f/\partial y)$ is the Jacobian ideal. (Denoting the composition $\C[x,y]\to H^0(\O_C)\to H^0(\O_C/J)$ by  $\epsilon\mapsto[\epsilon]$, to first order the construction associates $[\epsilon]\in\mathsf{Def}$ to the deformation
$\{f+\epsilon=0\}$ of $\{f=0\}$.) Inside $\mathsf{Def}$ is the \emph{equigeneric locus} of deformations
with the same geometric genus; it has tangent cone $H^0(A/J)\subset H^0(\O_C/J)$ where $A\subset\O_C$ is the \emph{conductor ideal}  \cite{T} of colength $g-\bar g$. Inside that is the equisingular locus, which is smooth with tangent space $H^0(ES/J)\subset H^0(\O_C/J)$ where $ES\subset\O_C$ is the \emph{equisingular ideal}. The inclusion $ES\subset A$ is \emph{strict} unless the singularities are nodal.

First we show that the locus of $\delta$-nodal curves in $\PP:=\PP(H^0(L))$ is smooth of codimension $\delta$, from which the first result of the Proposition follows. Fix $[s]\in\PP$ corresponding to a $\delta$-nodal curve $C\subset S$. The germ of $\PP$ about
$[s]$ maps to the miniversal deformation space $\mathsf{Def}$ of the singularities of $C$.
Its tangent map $T_{[s]}\PP=H^0(L)/\langle s\rangle\to H^0(L\otimes\O_C/J)$ is onto because length$(\O_C/J)=\delta$ and $L$ is $(\delta-1)$-very ample.
Therefore this is a smooth map. Inside $\mathsf{Def}$ the locus of $\delta$-nodal curve germs is smooth of codimension $\delta$, so the locus of $\delta$-nodal curves in $\PP$ is also smooth  of codimension $\delta$.

Now we show in turn that curves of the following kind
\begin{itemize}
\item nonreduced,
\item reduced with geometric genus $<g-\delta$, and
\item reduced with geometric genus $g-\delta$ with singularities other than nodes,
\end{itemize}
sit in subschemes of $\PP$ of codimension $>\delta$.

$\bullet$ Fix a nonreduced curve $C$ cut out by $s\in H^0(L)$, with underlying reduced curve $C^{red}$. Nearby curves in $\PP$ are of the form $s+\epsilon=0$ for small $\epsilon$ in a fixed complement to $\langle s\rangle\subset H^0(L)$. A local computation shows that where $\epsilon|_{C^{red}}\ne0$, the resulting curve is reduced. So the tangent cone to the nonreduced curves lies in the kernel of $T_{[s]}\PP=H^0(L)/\langle s\rangle\to H^0(L|_{C^{red}})$. This map need not be onto, but its composition with the projection to any length-$(\delta+1)$ subscheme of $C^{red}$ is, since $L$ is
$\delta$-very ample. Therefore its kernel has codimension at least $\delta+1$, so the nonreduced curves have codimension $>\delta$ inside $\PP$, as required.

$\bullet$ Given a curve $C$ of geometric genus $\bar g<g-\delta$ defined by $s\in H^0(L)$, the tangent cone to the
equigeneric locus of curves of
the same geometric genus in $\PP(H^0(L))$ is given by the kernel of $H^0(L)/\langle
s\rangle\to H^0(L\otimes(\O_C/A))$. This map need not be onto, but its composition with the projection
to any length-$(\delta+1)$ subscheme of $\O_C/A$ is, since $L$ is
$\delta$-very ample. Therefore the kernel has codimension at least $\delta+1$, and the equigeneric locus has
codimension $>\delta$.

$\bullet$  Finally we deal with reduced curves of geometric genus $g-\delta$ which are \emph{not} $\delta$-nodal.
Let $S_m\subset\PP$ be the locus of curves of geometric genus $g-\delta$ whose total Milnor number
(summed over the singular points of $C$) is $m$. For plane curve singularities, it is a classical fact
that $m\ge\delta$ with equality only for the $\delta$-nodal curves. Thus $S_\delta$ is the locus of nodal curves,
which we have observed is smooth of codimension $\delta$.

For $m>\delta$ we will show that $S_m$ is in the closure of
$S_\delta\cup S_{\delta+1}\cup\cdots\cup S_{m-1}$ and so has codimension $\ge\delta+1$ by induction, as required.
In fact it will be enough to prove that for each $C\in S_m$, some equigeneric deformation of $C$ is not equisingular, and so, by
a classical result \cite{LR}, has smaller Milnor number.

Since $L$ is $(\delta-1)$-very ample, $H^0(L)/\langle s\rangle$ surjects onto $H^0(L\otimes\O_C/A)$. Therefore the image of the germ of $\PP$  about $[s]$ is \emph{transverse} to the equigeneric locus of curves in $\mathsf{Def}$.
Since $C$ has non-nodal singularities, $A/ES$ has length $\ge1$, so there is a nonzero map $A/ES\to\O_p$, where $p\in C$ one of the singular points. Its kernel defines an ideal $I$ with
$J\subset ES\subset I\subset A\subset\O_C$ and $A/I$ one dimensional.
By the $\delta$-very ampleness of $L$,  $H^0(L)/\langle s\rangle$ surjects onto $H^0(L\otimes\O_C/I)$. Therefore $\PP$ maps onto
an equigeneric 1-dimensional family of deformations of $C$ whose tangent cone in $H^0(C,L\otimes A/J)$ is
transverse to $H^0(C,L\otimes I/J)$ (which by construction contains the tangent space to the smooth locus of equisingular deformations). So we get our equigeneric family in which the Milnor number drops, as required.
\end{proof}

\section{BPS calculus}

\begin{prop} \emph{\cite[Appendix B.1]{PT3}} \label{1}
For $C$ a reduced curve embedded in a smooth surface, the generating series of Euler characteristics
$$
q^{1-g}\sum_{i=0}^\infty e\big(C^{[i]}\big)q^i
$$
can be written uniquely in the form
$$
\sum_{r=g(\bar C)}^{g(C)}n_{r,C\ }q^{1-r}(1-q)^{2r-2}
$$
for integers $n_{r,C},\ r=\bar g,\ldots,g$.

Moreover, the $n_{r,C}$ are determined by only the numbers $e\big(C^{[i]}\big),\ i=1,\ldots,g-\bar
g$ by a linear relation. Explicitly, $n_{r,C}=0$ for $r>g$ and $n_{g,C}=1$, while for
$r<g$ the $n_{r,C}$ are determined inductively by
$$
n_{g-r,C}=e\big(C^{[r]}\big)-\sum_{i=g-r+1}^gn\_{i,C}\ e(\Sym^{r-(g-i)}\Sigma_i),
$$
where $\Sigma_i$ is any smooth proper curve of genus $i$.  $\hfill\square$
\end{prop}

Proposition \ref{1} is a consequence of Serre duality for the fibres of the Abel-Jacobi
map taking a subscheme $Z\subset C$ to the sheaf $I_Z^*$. Its force is that these formulae give $n_r=0$ for all $r<\bar g$, \emph{regardless of
the singularities of $C$}.

The result says that at the level of Euler characteristics of Hilbert schemes, the singular
curve $C$ looks like a disjoint union of smooth curves $\Sigma_i$ of genera $i$ between $\bar g$ and $g$; the number of genus $i$ being $n_{i,C}$. (Note that the generating series for $\Sigma_i$ is precisely $q^{1-i}(1-q)^{2i-2}$. The example of nodal
curves in Proposition
\ref{2} below is illustrative.) It is most natural and comprehensible in the context
of stable pairs \cite{PT3} (which are dual to the ideal sheaves parametrised
by the Hilbert scheme), but a self contained account using only Hilbert schemes is given
in \cite[Proposition 2]{Sh}. Moreover, it is shown there that the $n_{r,C}$ are \emph{positive}.

\begin{prop} \emph{\cite[Proposition 3.23]{PT3}} \label{2}
For $C$ an $\delta$-nodal curve,
$$
n_{r,C\ }={\delta\choose g-r}
$$
is the number of partial normalisations of $C$ at $g-r$ of its $\delta$ nodes -- i.e. the
number of curves of arithmetic genus $r$ mapping to $C$ by a finite map which is generically an isomorphism. In particular, $n_{g-1,C}=\delta$ is the number of nodes, while
$n_{g,C}=1=n_{g-\delta,C}$ and $n_{i,C}=0$ for $i\not\in\{g-\delta,\ldots,g\}$.
$\hfill\square$
\end{prop}

To give some feeling for these two results, we explain how the resulting formulae for
a $\delta$-nodal curve $C$,
\beq{BPS1}
e\big(C^{[k]}\big)=\sum_{j=0}^k{\delta\choose j}e(\Sym^{k-j}\Sigma_{g-j}),
\eeq
arise from a decomposition of $C^{[k]}$ into Zariski locally closed subsets. The first piece is $\Sym^kC_0\subset C^{[k]}$, where $C_0=C\take\{\mathrm{nodes}\}$ is the smooth locus. Since
$$
e(C_0)=e(\Sigma_g) \quad\mathrm{and\ so}\quad e(\Sym^kC_0)=e(\Sym^k\Sigma_g),
$$
this gives the first term of \eqref{BPS1}. To describe the $j$th piece of the decomposition, fix $j$ distinct nodes $p_{i_1},\cdots,p_{i_j}\in C$ at which we normalise to produce an arithmetic genus $g-j$ curve $\overline C^{i_1\cdots i_j}$ with $\delta-j$ nodes and smooth locus $(\overline C^{i_1\cdots i_j})_0$. Then we have
$$
\Sym^{k-j}(\overline C^{i_1\cdots i_j})_0\into C^{[k]}
$$
by projecting the $(k-j)$-cycle in $(\overline C^{i_1\cdots i_j})_0$ down to $C$ and adding in the nodes $p_{i_1},\cdots,p_{i_j}$. Around the normalised nodes $p_{i_l}$ we have to explain what we mean by this. Describing the node $p_{i_l}$ locally as $\{xy=0\}\subset\C^2$, the partial normalisation has two local branches corresponding to the $x$- and $y$- axes. If the $(k-j)$-cycle has multiplicities $a$ and $b$ at the two preimages of the node on these branches, then we push this down to the length-$(a+b+1)$ subscheme with local ideal $(x^{a+1},y^{b+1})$. That is, we thicken $p_{i_l}$ further by $a$ units down the $x$-axis and $b$ units down the $y$-axis.

Taking the disjoint union over all $n_{g-j,C}$ choices of the nodes $p_{i_1},\cdots,p_{i_j}$ gives the $j$th piece. Since
$$
e((\overline C^{i_1\cdots i_j})_0)=e(\Sigma_{g-j}) \quad\mathrm{and\ so}\quad e(\Sym^k(\overline C^{i_1\cdots i_j})_0)=e(\Sym^k\Sigma_{g-j}),
$$
its Euler characteristic gives the $j$th term in \eqref{BPS1}. What we are missing of $C^{[k]}$ in this decomposition is those subschemes which are Cartier divisors with nonempty intersection with the set $\{p_1,\cdots,p_\delta\}$ of nodes. The set of such subschemes has a $\C^*$-action (with weights $(1,0)$ on $\C^2\supset\{xy=0\}$) with no fixed points, and so zero Euler characteristic.

A simple consequence of Propositions \ref{1} and \ref{2} is the following.

\begin{thm} \label{BPS}
Fix a general linear system $\PP^\delta$ as in Proposition \ref{Got}, with universal curve
$\cC\to\PP^\delta$. Then the number of $\delta$-nodal curves in $\PP^\delta$ is a linear combination of the numbers $e(\Hilb^i(\cC/\PP^\delta)),\ i=1,\ldots,\delta$.

It is the coefficient
$n_{g-\delta}$ of $q^{1-g+\delta}(1-q)^{2g-2\delta-2}$ in the generating series
$$
q^{1-g}\sum_{i=0}^\infty e(\Hilb^i(\cC/\PP^\delta))q^i=
\sum_{r=g-\delta}^g n_r\,q^{1-r}(1-q)^{2r-2}.
$$
\end{thm}

\begin{proof}
Compute the Euler characteristics fibrewise, stratifying $\PP^\delta$ by the topological
type of the curve. By Proposition \ref{1}, $n_{g-\delta,C}$ vanishes for all but the curves $C$ in our linear
system with geometric genus $g-\delta$. By Proposition \ref{Got} these are the $\delta$-nodal
curves, for which $n_{g-\delta,C}=1$ by Proposition \ref{2}. Therefore the only contributions to the coefficient
$n_{g-\delta}$ of the above generating series are precisely 1 for each $\delta$-nodal curve.
\end{proof}

\section{Tautological integrals}
These relative Hilbert schemes are relatively easy to compute with. Let $\PP:=\PP(H^0(L))$
denote the complete linear system. The sections $H^0(L)\otimes H^0(L)^*$ of $L\boxtimes\O(1)$ on $S\times\PP$ include the canonical section $\id_{H^0(L)}$. Pull it back to
$S\times S^{[i]}\times\PP$, restrict to the universal subscheme $\mathcal Z_i\times\PP$, and push down to
$S^{[i]}\times\PP$ to give the tautological section of $L^{[i]}\boxtimes\O(1)=:L^{[i]}(1)$.
(Thinking of it as a section of $\hom(\O(-1),L^{[i]})$ on $S^{[i]}\times\PP$, at a point
$(Z,[s])$ it maps $s\in\O(-1)$ to $s|_Z\in H^0(L|_Z)=L^{[i]}|_Z$.)

This tautological section has scheme theoretic zero locus
$$
\Hilb^i(\cC/\PP)\subset S^{[i]}\times\PP.
$$
This is smooth, fibring over $S^{[i]}$ with fibres $\PP\big(\!\ker(
H^0(L)\to H^0(L|_Z)\big)$ of constant rank, because $L$ is $(i-1)$-very ample. Since its codimension equals the
rank $i$ of $L^{[i]}(1)$, the tautological section was in fact transverse and
$[\Hilb^i(\cC/\PP)]$ is Poincar\'e dual to $c_i(L^{[i]}(1))$. Intersecting general hyperplanes in $\PP$, by Bertini we retain smoothness when $\PP$ is replaced by a general $\PP^{\delta}$.

It follows that the Euler characteristic
of the relative Hilbert scheme -- the integral of $c_\bull(T\Hilb^i(\cC/\PP^\delta))$
over $\Hilb^i(\cC/\PP^\delta)$ -- can be computed as
$$
\int_{\Hilb^i(\cC/\PP^\delta)}
\frac{c_\bull(T(S^{[i]}\times\PP^\delta))}{c_{\bull}(L^{[i]}(1))}\ =\,
\int_{S^{[i]}\times\PP^\delta}c_i(L^{[i]}(1))
\frac{c_\bull(T(S^{[i]}\times\PP^\delta))}{c_{\bull}(L^{[i]}(1))}\,.
$$

\begin{thm}
Suppose that $L$ is $\delta$-very ample. Then in a
$\PP^\delta$ linear system, general in the sense of Proposition \ref{Got},
the number of $\delta$-nodal curves is a polynomial of degree $\delta$ in $L^2,\,L.K_S,\,K_S^2$ and $c_2(S)$.
\end{thm}

\begin{proof}
By Proposition \ref{Got}, we want to compute the intersection of $\PP^\delta$ with the quasiprojective variety of $\delta$-nodal curves in $\PP$. Thus we may perturb $\PP^\delta$ without changing this number to 
make it general in the sense of both Proposition \ref{Got} and the Bertini theorem above. So we can assume that $\Hilb^i(\cC/\PP)$ is smooth.

By the above computation, $e(\Hilb^i(\cC/\PP^\delta))$ is then the integral over $S^{[i]}$ of a polynomial in the Chern classes of $L^{[i]}$ and $S^{[i]}$: the coefficient of $\omega^\delta$ in
\beq{mess}
\frac{c_\bull(T(S^{[i]}))(1+\omega)^{\delta+1}\sum_{j=0}^i\omega^jc_{i-j}(L^{[i]})}
{\sum_{j=0}^i(1+\omega)^jc_{i-j}(L^{[i]})}\,.
\eeq
The recursion of \cite{EGL} applied $i$ times turns this into an integral over $S^i$ of a polynomial in $c_1(L),\,c_1(S),\,c_2(S)$ (pulled back from different $S$ factors)
and $\Delta_*1,\,\Delta_*c_1(S),\,\Delta_*c_1(S)^2,\,\Delta_*c_2(S)$ (pulled back from different $S\times
S$ factors), where $\Delta\colon S\into S\times S$ is the diagonal. The result is a degree
$\le i$ polynomial in $L^2,\,L.K_S,\,K_S^2$ and $c_2(S)$. 

By Theorem \ref{BPS} the number $n_{g-\delta}$ of $\delta$-nodal curves is a linear combination
of these degree $\le i$ polynomials, for $0\le i\le\delta$. Their coefficients are polynomials of degree $\le\delta-i$ in $g=1+(L^2+L.K_S)/2$. Thus it has total degree $\le\delta$.
In fact one easily sees from \cite{EGL} that our integral over $S^{[i]}$ produces a single term in $c_2(S)^i$ (the coefficient being $(\delta-i+1)/i!$) arising from the $(\delta-i+1)c_{2i}(T(S^{[i]}))$ term in \eqref{mess}. Therefore $n_{g-\delta}$ contains a term $c_2(S)^\delta/\delta!$ which does not cancel with any other terms, so the polynomial
has degree precisely $\delta$.
\end{proof}

The recursion is constructive, though not terribly efficient. Our computer programme only calculates up to 4 nodes without trouble, finding agreement with \cite{KP1, V} (after correcting a small sign error in \cite{EGL}).



\begin{thebibliography}{MNOP}

\bibitem[Be]{Be} A. Beauville,
{\em Counting rational curves on $K3$ surfaces}, Duke Math. Jour. \textbf{97}, 99--108,
1999. alg-geom/9701019. 

\bibitem[BL1]{BL1}
J. Bryan and N. C. Leung,
{\em The enumerative geometry of K3 surfaces and modular forms},
Jour. AMS \textbf{13}, 371--410, 2000. alg-geom/9711031.

\bibitem[BL2]{BL2}
J. Bryan and N. C. Leung,
{\em Generating functions for the number of curves on abelian surfaces},
Duke Math. Jour. \textbf{99}, 311--328, 1999. math.AG/9802125.

\bibitem[CH]{CH}
L. Caporaso and J. Harris,
{\em Counting plane curves of any genus}, Invent. Math. \textbf{131}, 345--392, 1998.
alg-geom/9608025.

\bibitem[Ch]{Ch}
Y. Choi,
{\em Enumerative Geometry of Plane Curves}, PhD thesis, University of California, Riverside,
1999.

\bibitem[DFI]{DFI}
P. Di Francesco and C. Itzykson, {\em Quantum intersection rings}, in \emph{The Moduli
Space of Curves}, Birkh\"auser, 81--148, 1995. hep-th/9412175.

\bibitem[DH]{DH}
S. Diaz and J. Harris, {\em Ideals associated to deformations of singular plane curves},
Trans. AMS \textbf{309}, 433--468, 1988.

\bibitem[EGL]{EGL}
G. Ellingsrud, L. G\"ottsche and M. Lehn,
{\em On the cobordism class of the Hilbert scheme of a surface},
Jour. Alg. Geom. \textbf{10}, 81--100, 2001. math.AG/9904095.

\bibitem[FM]{FM}
S. Fomin and G. Mikhalkin,
{\em Labeled floor diagrams for plane curves}, arXiv:0906.3828.

\bibitem[G]{G}
L. G\"ottsche,
{\em A conjectural generating function for numbers of curves on surfaces},
Comm. Math. Phys. \textbf{196}, 523--533, 1998. math.AG/9808007.

\bibitem[KY]{KY}
T.~Kawai and K.~Yoshioka,
\newblock {\em String partition functions and infinite products}, 
{Adv. Theor. Math. Phys.}, {\bf 4}, 397--485, 2000. hep-th/0002169.

\bibitem[Ka]{Ka}
M. Kazarian, {\em Multisingularities, cobordisms, and enumerative geometry},
Russ. Math. Surv. \textbf{58}, 665--724, 2003.

\bibitem[KP1]{KP1}
S. Kleiman, R. Piene,
{\em Enumerating singular curves on surfaces}, Cont. Math. \textbf{241}, 209--238, 1999. math.AG/9903192.

\bibitem[KP2]{KP2}
S. Kleiman, R. Piene,
{\em Node polynomials for families: methods and applications}, Math. Nachr. \textbf{271}, 69--90, 2004. math.AG/0111299.

\bibitem[KT]{KT}
M.~Kool and R.~P. Thomas,
\newblock {\em Reduced classes and curve counting on surfaces}, in preparation.

\bibitem[LR]{LR}
L\^e D\~ung Tr\'ang and C. P. Ramanujam,
{\em The invariance of {M}ilnor's number implies
the invariance of the topological type}, Am. Jour. Math. \textbf{98}, 67--78, 1976.

\bibitem[L1]{Liu1}
A.~Liu, {\em Family blowup formula, admissible graphs and the enumeration of singular curves, I}, Jour. Diff. Geom. \textbf{56}, 381--579, 2000.

\bibitem[L2]{Liu2}
A.~Liu, {\em The Algebraic Proof of the Universality Theorem}, math.AG/0402045.

\bibitem[MNOP]{MNOP1}
D.~Maulik, N.~Nekrasov, A.~Okounkov, and R.~Pandharipande.
\newblock {\em Gromov-{W}itten theory and {D}onaldson-{T}homas theory. {I}},
  Compos. Math., {\bf 142}, 1263--1285, 2006.
\newblock math.AG/0312059.

\bibitem[MPT]{MPT} 
D.~Maulik, R.~Pandharipande and R.~P. Thomas,
\newblock{\em Curves on K3 surfaces and modular forms}, with an Appendix by A. Pixton.
To appear in Jour. Topol. arXiv:1001.2719.

\bibitem[PT1]{PT1}
R.~Pandharipande and R.~P. Thomas,
\newblock {\em Curve counting via stable pairs in the derived category},
Invent. Math. \textbf{178}, 407--447, 2009.
\newblock arXiv:0707.2348.

\bibitem[PT3]{PT3} 
R.~Pandharipande and R.~P. Thomas,
\newblock{\em Stable pairs and BPS invariants},
Jour. AMS \textbf{23}, 267--297, 2010. \newblock arXiv:0711.3899.

\bibitem[R]{R} 
Z. Ran,
{\em Enumerative geometry of singular plane curves}, Invent. Math. \textbf{97}, 447--469,
1989.

\bibitem[Sh]{Sh} 
V.~Shende,
\newblock{\em Hilbert schemes of points on a locally planar curve and the Severi strata of its versal deformation}, \newblock arXiv:1009.0914.

\bibitem[T]{T}
B. Teissier,
{\em R\'esolution simultan\'ee: I -- Familles de courbes}, in S\'eminaire sur les singularit\'es des surfaces, Springer LNM \textbf{777}, 1980.

\bibitem[Tz]{Tz}
Y.~Tzeng,
{\em A proof of the G\"ottsche-Yau-Zaslow formula}, arXiv:1009.5371.

\bibitem[V]{V} 
I. Vainsencher,
{\em Enumeration of n-fold tangent hyperplanes to a surface},
Jour. Alg. Geom. \textbf{4}, 503--526, 1995. alg-geom/9312012.

\bibitem[YZ]{YZ} 
S.-T. Yau and E. Zaslow, {\em BPS states, string duality, and nodal curves on $K3$},
Nucl. Phys. B\textbf{471}, 503--512, 1996.hep-th/9512121.

\end{thebibliography}
\end{document}